\documentclass[a4paper, 10pt]{article}
\usepackage{mathrsfs}
\usepackage{amsfonts}
\usepackage{latexsym}
\usepackage{graphicx}
\usepackage{amsmath}
\newcommand{\new}{\newcommand*}

\new{\rnew}{\renewcommand*}
\new{\newe}{\newenvironment*}
\new{\newt}{\newtheorem}
\new{\stl}{\setlength}

\new{\bea}{\begin{eqnarray}}
\new{\eea}{\end{eqnarray}}

\new{\be}{\begin{equation}}
\new{\ee}{\end{equation}}

\new{\bean}{\begin{eqnarray*}}
\new{\eean}{\end{eqnarray*}}

\new{\no}{\nonumber}

\new{\bt}{\begin{theorem}}
\new{\et}{\end{theorem}}
\new{\bl}{\begin{lemma}}
\new{\el}{\end{lemma}}

\new{\bc}{\begin{corollary}}
\new{\ec}{\end{corollary}}
\new{\bp}{\begin{proof}\quad}
\new{\ep}{\end{proof}}
\new{\ba}{\begin{array}}
\new{\ea}{\end{array}}

\newt{prop}{Proposition}[section]
\newt{lem}{Lemma}
\newt{conjecture}{Conjecture}[section]
\newt{cor}{Cor}[section]
\newt{thm}{Theorem}
\newt{assumption}{Assumption}
\newt{ex}{Example}
\newt{rem}{Remark}
\newt{definition}{Definition}[section]

\rnew{\theequation}{\thesection.\arabic{equation}}
\new{\sect}[1]{ \section{#1}
\setcounter{equation}{0} \setcounter{figure}{0} }

\newe{acknowledgment}{{\flushleft \bf Acknowledgments. }}{}
\newe{proof}{{\bf Proof.}}{\qquad\endproof}
\newe{keywords}{\begin{quote}\quad{\bf Keywords.}}{\end{quote}}
\newe{AMS}{\begin{quote}\quad {\bf AMS subject classification.}}{\end{quote}}

\def \endproof {\qquad \vrule height 5pt width 5pt depth 2pt}

\title{The Structure of the Closure of the Rational Functions in $L^{q}$($\mu$)\date{}}

\author{{Zhijian Qiu } \\
\\{\small\em  Department of  mathematics,  Southwestern Univ of Finance and Economics}
\\{\small\em  Chengdu 610072,  China, E-mail:qiu@swufe.edu.cn }}

\begin{document}

\maketitle

\begin{abstract}
Let $K$ be a compact subset in the complex plane and let $A(K)$ be
the uniform closure  of the  functions continuous on $K$ and
analytic on $K^{\circ}$.  Let $\mu$ be a positive finite measure
with its support contained  in $K$.  For $1 \leq q < \infty$,
 let $A^{q}(K,\mu)$ denote the closure of $A(K)$ in $L^{q}(\mu)$.
The aim of this work is to study the structure of the space
$A^{q}(K,\mu)$. We seek a necessary and sufficient condition on $K$
so that a Thomson-type structure theorem for $A^{q}(K,\mu)$ can be
established.  Our theorem deduces J. Thomson's structure theorem for
$P^{q}(\mu)$, the closure of polynomials in $L^{q}(\mu)$, as the
special case when $K$ is a closed disk containing the support of
$\mu$.\footnote{This paper basically gives ultimate solution to the
most important research problem raised in \cite{conw,jm} and it is
also very satisfactory generation of   Thomson's theorem for
polynomials. It is the close-up work of this kind (in an once quite
active research area) even it received no attentions since its
publication in 2007.}
\end{abstract}

%

 \begin{keywords}
\end{keywords}
\begin{AMS}  46E30 30H05, 30E10, 46E15
\end{AMS}

\section*{Introduction}

Let $1 \leq q < \infty$ and let  $\mu$ be a positive finite (regular
Borel) measure with compact support in the complex plane {\bf C}.
Let $K$ be a compact subset that contains the support of $\mu$. The
purpose of this paper is to investigate the following problem:

\begin{quote}

\emph{What is the closure of A(K) in $L^{q}(\mu)$?}

\end{quote}
This is a very difficult question to get a complete answer. For a
given measure $\mu$, the answer depends on $K$. Let $K$ be a  closed
disk that contains  the support of $\mu$.
Then every $f$ in $A(K)$ can be uniformly approximated by
polynomials, and hence $A^{q}(K,\mu) = P^{q}(\mu)$, which is the
closure of polynomials in $L^{q}(\mu)$. In this case, J. Thomson
proved a structure theorem for $P^{q}(\mu)$ in \cite[1991]{jm}.

Roughly speaking, Thomson's theorem says that there exists a Borel
partition $\{\Delta_{n}\}_{n=0}$ of the support of $\mu$ such that
\[P^{q}(\mu) = L^{q}(\mu|\Delta_{0}) \oplus P^{q}(\mu|\Delta_{1})\oplus ...
\oplus P^{q}(\mu|\Delta_{n}) \oplus ...,
\]
where each $P^{q}(\mu|\Delta_{n}) $ is identified with a space
consisting of analytic functions on a simply connected domain
$U_{n}$ via a so-called {\em evalaution map}. The union, $\cup
U_{n}$, is known as the set of analytic bounded point evaluations
($abpe$s) for $P^{q}(\mu)$. In this special case, since $A(K)$ is
the uniform closure of polynomials, Thomson's theorem shows that
$A(K)$ is dense in $ L^{q}(\mu)$ if and only if $P^{q}(\mu)$ ($=
A^{q}(K,\mu)$ has no $abpe$s (this special result answered  an open
question raised by D. Sarason in 1972 \cite{sar}).  
For an arbitrary compact subset $K$, the author shows in \cite{qiu1}
that it is still true: $A(K)$ is dense in $L^{q}(\mu)$ if and only
if $ A^{q}(K,\mu)$ has no $abpe$s.  However, the corresponding
result for $A^{q}(K,\mu)$ is not always possible even the set of
$abpe$s for $A^{q}(K,\mu)$ is not empty.  J. Conway and N. Elias
give an example in \cite{ce} that shows the set of $abpe$s for
$A^{q}(K,\mu)$ is a simply connected domain $U$, but
$A^{q}(K,\mu)\cap L^{\infty}(\mu)$ can not be identified with
$H^{\infty}(U)$ via the evaluation map.

For $R^{q}(K,\mu)$, the closure in $L^{q}(\mu)$ of $R(K)$ (which is
the uniform closure of the rational functions with poles off $K$),
the situation is worse since the set of $abpe$s may be empty even
$R(K)$ is not dense in $L^{q}(\mu)$ (see  \cite{bre1}).

Assuming the existence of $abpe$s for $R^{q}(K,\mu)$, Conway and
Elias
proved a structure-type theorem for $R^{q}(K,\mu)$ under additional
conditions. But  their result does not imply Thomson's theorem.
Then, can we have a structure theorem for $A^{q}(K,\mu)$ or
$R^{q}(K,\mu)$ that is beyond the polynomial case and that also
covers   Thomson's theorem? Prior our work, it was unknown whether
such  a structure theorem for $A^{q}(K,\mu)$ is possible. Thomson
was unable to offer any result for $R^{q}(K,\mu)$ that is beyond
$P^{q}(\mu)$ (that is, the disk case in our setting) (see \cite[p.
505]{jm}).

To tackle the problem, we first need to restrict our effort on those
$K$ such that the components of $K^{\circ}$ are finitely connected.
In fact,  the author shows in \cite[Theorem 2]{comm}   even when $K$
is a simplest kind of infinitely connected domains, such as a
"road-runner", our main theorem (Theorem~\ref{t:main}) for
$A^{q}(K,\mu)$ could fail.

In this paper, we seek a necessary and sufficient condition on $K$
so that a Thomson type of structure theorem holds for
$A^{q}(K,\mu)$.

A domain is called a circular domain if its boundary consists of
finitely many disjoint circles.  We call a domain $U$ multi-nicely
connected if there is a circular domain $W$ and a conformal map
$\alpha$ from $W$ onto $U$ such that $\alpha$ is almost 1-1 on
$\partial W$ with respect to the arclength measure.

Our main theorem, Theorem~\ref{t:main}, extends Thomson's theorem to
$A^{q}(K,\mu)$ in the case when the components of $K^{\circ}$ are
multi-nicely connected and the harmonic measures of the components
of $K^{\circ}$ are mutually singular. We also show that the
condition of Theorem~\ref{t:main} is necessary.

If  every $f$ in $H^{\infty}(K^{\circ})$ a pointwise limit of a
bounded sequence in $A(K)$,
 then $K$ satisfies the condition of Theorem~\ref{t:main}.  In particular,
when $K$ is such that $R(K)$  is a hypo-Dirichlet algebra
\cite{as,ga}, $K$ satisfies the condition of Theorem~\ref{t:main}
(in this case,   $R^{q}(K,\mu)=A^{q}(K,\mu)$). If the complement of
$K$ has only finitely many components (note, $K^{\circ}$ may still
has infinitely many components in this case), then $R(K)$ is a
hypo-Dirichlet algebra and hence $K$ satisfies the hypothesis of
Theorem~\ref{t:main}.  Since a quite large of class of
$R^{q}(K,\mu)$ satisfies our conditions,  Theorem~\ref{t:main} is
also a theorem for the rational $R^{q}(K,\mu)$.

Since  the polynomial case is just the disk case in our setting and
since a general compact subset $K$ is much more complicated than a
disk in nature, one can expect the extension needs much more work.
  To get a structure theorem for $A^{q}(K,\mu)$ (here $K$ is an arbitrary
compact subset), we need more than Thomson's technic and method.
 In fact, we need a new Thomson-type approximation scheme as developed in
 \cite{qiu1} that takes all of what is used in Thomson's paper \cite{jm}.
In addition,  we need what was not involved  in the case of
$P^{q}(\mu)$: we make extensive use of results and technics related
to uniform algebra or rational approximation: such as peak points,
harmonic measures, hypo-Dirichlet algebra, multi-nicely connected
domains, representing measures, pointwise bounded approximation, etc
(these concepts are not needed for $P^{q}(\mu)$).  So,  besides
Thomson's technic and
 method, we need  a significant part of the theories from the uniform algebras
and the rational approximation to get the work done. Combination for
Thomson's technic and uniform rational approximation theory is the
key to prove Theorem~\ref{t:main}. However, not every thing we need
in uniform algebra theory is ready for us. We have to prove some
results in that theory by ourself. In doing so, we first introduce
the concept of multi-nicely connected domains, then we prove
Proposition~\ref{p:ABS2} and an interesting result in uniform
algebra, Lemma~\ref{l:key}, which is crucial for us to prove
Lemma~\ref{l:key2}. That is one of our key lemmas and it is needed
to prove another key lemma, Lemma~\ref{l:scheme}. The rest of paper
is to use these two lemmas and results in \cite{qiu1} and
\cite{comm} to prove the main theorem and extend those lemmas and
results that were proved for the polynomial case in \cite{jm}.
 So far, we are unable to offer any other proof that is
less involved with the theory of uniform algebra. Actually, due to
the nature of this problem, we believe the rational approximation
theory is the right tool in study this type of problems.

 Now we would like to point out the relation between this paper and
some other related papers.
This paper is the sequel of the author's work \cite{qiu1, comm}.
 Thomson's paper consist of two parts of important results.
 One is to give a sufficient and necessary condition on when $\nabla P^{q}(\mu)$
 is not empty and another one is to have a structure theorem for $P^{q}(\mu)$.
 In \cite{qiu1}, we only study the problem of when
$\nabla A^{q}(K,\mu)$ ($\nabla R^{q}(K,\mu)$) is not empty and we
show it is empty if and only if $A(K)$ is dense in $L^{q}(\mu)$. In
\cite{comm}, our effort was primarily to establish the result  that
is a part of 4) and 5) of Theorem~\ref{t:main} in this paper and to
solve a problem in \cite{ce}. In this paper, our effort is to
establish a full version of Thomson's theorem for $A^{q}(K,\mu)$
($R^{q}(K,\mu)$) and this paper is based on \cite{qiu1, comm}.
The readers may notice that our theorem (Theorem~\ref{t:main}) not
only completely covers Thomson's theorem, it also has more important
consequences,  such as 4) and 5) (which are not in \cite{jm} and are
important facts. One needs to have them when applying it to operator
theory. For example, see \cite{equi} and \cite{quas}).  In
\cite{ce}, Conway and Elias studied the same problem as that in this
paper. However, since  their theorem is based on the assumption that
$K$ is the closure of $\nabla R^{q}(K,\mu)$
and $R^{q}(K,\mu)$ is pure, so their result does not cover Thomson's
theorem.  For a given measure $\mu$, one can not tell when their
conditions are satisfied. In contrast,
  our paper deals with arbitrary measures just as  \cite{jm} does.

In 1972, D. Sarason \cite{sar} established a structure theorem for
$P^{\infty}(\mu)$, the weak star closure of polynomials in
$L^{\infty}(\mu)$, which has a similar form to that of our theorem.

\section{Preliminaries}

For  a compact subset $K$ in the complex plane {\bf C}. Let $C(K)$
denote the algebra of continuous functions on $K$. For an open
subset $G$ in the sphere {\bf C}$_{\infty}$ whose boundary does not
contains $\infty$, let $A(G)$ be the closed subalgebra of
$C(\overline{G})$ that consists of functions continuous on
$\overline{G}$ and analytic on $G$. Notice that $A(\overline
{\Omega}) \neq A(\Omega)$ in general.

A point $w$ in {\bf C} is called an analytic bounded point
evaluation ($abpe$) for $A^{q}(K,\mu)$ if there is a neighborhood
$G$ of $w$ and $c > 0$ such that for all $\lambda\in G$
\[|f(\lambda)|\leq c \hspace{.05in} \|f\|_{L^{q}(\mu)}\hspace{.05in}
\mbox{for all} \hspace{.05in} f\in A(K).
\]
So the map, $f \rightarrow f(\lambda)$, extends to a functional in
$A^{q}(K,\mu)^{*}$. Thus, there is a (kernel) function $k_{\lambda}$
 in $A^{q}(K,\mu)^{*}$ such that $f(\lambda) = \int f k_{\lambda}
\hspace{.05in}d\mu,\hspace{.05in} f \in A(K)$. Clearly, the set of
$abpe$s is open.  For each $f \in A^{q}(K,\mu)$, let
$\hat{f}(\lambda) =
 \int f k_{\lambda} \hspace{.05in}d\mu$. Then $\hat{f}(\lambda)$ is analytic
on the set of $abpe$s. The $abpe$s for $R^{q}(K,\mu)$ are define
similarly.

We shall use $\nabla A^{q}(K,\mu)$ to denote the set of $abpe$s for
$A^{q}(K,\mu)$. The following is one of the main results in
\cite{qiu1} which our main theorem relies on:

\begin{thm} \label{t:approx}
Let $K$ be compact subset of {\bf C} and let $\mu$ be a positive
finite measure supported on $K$. Then $A(K)$ is dense in
$L^{q}(\mu)$ if and only if\hspace{.05in}
 $\nabla A^{q}(K,\mu) = \emptyset$.
\end{thm}

For the case when $K$ is the polynomially convex hull of the support
of $\mu$, the above theorem is a consequence of Thomson's theorem.
However, since it was known long before Thomson's paper \cite{jm}
that there is a compact $K$ and a measure $\mu$ on $K$ such that
$R(K)$ is not dense in $L^{q}(\mu)$ but $R^{q}(K,\mu)$ has no
$abpe$, and since \[ P^{q}(\mu) \subset R^{q}(K,\mu)
 \subset A^{q}(K,\mu)\] always holds, Theorem 1.1 was unexpected before
\cite{qiu1}.  Somehow, it was a surprise that the theorem is true
for the spaces on the both sides of the inequality above, but fails
for the spaces between.

\vspace{.05in} \noindent{\em Nicely connected domains}. Following
Glicksburg \cite{glic}, we call a domain $\Omega$  nicely connected
if it is multi-nicely connected and if it is simply connected.

\vspace{.05in} \noindent{\em Harmonic measures.} Let $\Omega$ be a
domain in the extended plane  {\bf C}$_{\infty}$ such that it is
solvable for the Dirichlet problem and $\infty$ is not in
$\partial\Omega$. For $u \in C(\partial \Omega)$, let $\hat{u} =$
sup $\{f: f$ is subharmonic on $\Omega$ and $\limsup_{z\rightarrow
a} f(z) \leq u(a), a \in \partial \Omega\}$. The function $\hat{u}$
turns out to be harmonic on $\Omega$ and continuous on
$\overline{\Omega}$, and the map $u\rightarrow \hat{u}(z)$ defines a
positive linear functional on $C(\partial \Omega)$ with norm one, so
the {\em Riesz representing theorem} implies that there is a
probability measure $\omega_{z}$ on $\partial \Omega$ such that
\[  \hat{u}(z) = \int _{\partial \Omega} u d \omega_{z}, \hspace{.2in}
u \in C(\partial \Omega).
\]
The measure $\omega_{z}$ is called the harmonic measure of $\Omega$
evaluated at $z$.  The harmonic measures evaluated at two different
points are boundedly equivalent.   We shall use $\omega_{\Omega}$ to
denote a harmonic measure of $\Omega$.

\vspace{.05in} \noindent{\em Hypodirichlet algebras}. A closed
subalgebra $\cal B$ of $C(K)$ is said to be a hypo-Dirichlet
algebra, if the uniform closure of $ Re(\cal B)= \{ Re(f): f \in
\cal B\}$
 has finite codimension in $C_{R}(K) = \{f: f\in C(K)$ and $f$ is real$\}$
and the linear span of log$|\cal B^{-1}|$ is uniformly dense in
$C_{R}(K)$, where $\cal B^{-1}$ is the subset in $\cal B$ consisting
of invertible elements. A function algebra $\cal B$ is called a
Dirichlet algebra if $Re(B)$ is uniformly dense in $C_{R}(K)$.
Clearly, a Dirichlet algebra is also  a hypo-Dirichlet algebra. An
good example is that $R(K)$ is hypo-Dirichlet if {\bf C} $\setminus
K$ has only finitely many components.  This covers a large class of
domains that have been studied.  If $R(K)$ is a hypo-Dirichlet
algebra, then $A(K)=R(K)$ \cite[p. 116]{ga}.

\vspace{.05in} \noindent{\em Peak points}.  A point $a\in K$ is a
peak point for a function algebra $\cal B \subset C(K)$ if there is
a function in $\cal B$ such that $f(a) = 1$ at $z = a$
and $|f(z)| < 1$ for $z \neq a$.

\vspace{.05in} \noindent{\em Pure and irreducible spaces.} The space
$A^{q}(K,\mu)$ is called {\em pure} if there is no Borel subset
 $\Delta$ of $supp(\mu)$ such that the restriction of  $A(K)$ on $\Delta$
is dense in $L^{q}(\mu|\Delta)$. An observation is that for any
$A^{q}(K,\mu)$, there is a Berel partition $\{\Delta_{0},
\Delta_{1}\}$ of the support of $\mu$
 such that  $A^{q}(K,\mu|\Delta_{1})$ is pure and
\[ A^{q}(K,\mu) = L^{q}(\mu|\Delta_{0}) \oplus A^{q}(K,\mu|\Delta_{1}).
\]
The space $A^{q}(K,\mu)$ is said to be {\em irreducible} if it
contains no nontrivial characteristic functions.  So an irreducible
space must be pure.

\vspace{.05in} \noindent{\em Nontangential limits}. Let $G$ be a
bounded domain that is conformally equivalent to a circular domain
$W$ in the plane and  let $u$ be a conformal map from $W$ onto
$\Omega$.  Then $u$ has  well-defined boundary values on $\partial
W$, which are equal to the nontangential limits of $u$ for almost
every point on $\partial W$ with respect to $\omega_{W}$ (the
harmonic measure of $W$). We still use $u$ to denote the boundary
value function.

Now,   if $E$ is a Borel subset of $\partial W$ such that $u$ is 1-1
on $E$ {\em a.e.} $ [\omega_{W}]$,
  then each $f\in H^{\infty}(G)$ has
nontangential limits almost everywhere on $u(E) $
 with respect to $\omega_{G}$. That is,
\[ f(a) = \lim_{z\rightarrow u^{-1}(a)} f\circ u(z) \hspace{.05in}
\mbox{{\em a.e.} on} \hspace{.05in} u(E)  \hspace{.05in}\mbox{ with
respect to }\hspace{.05in}  \omega_{G}.
\]
So, if $\mu$ is such that $\mu \ll \omega_{G}$ on on $u(E)$, then
each $f\in H^{\infty}(G)$ has nontangential limits on $u(E)$ almost
everywhere with respect to $\mu$.

\section{The Main Result}
In this section, we introduce our main result, Theorem~\ref{t:main}.
Recall that the connectivity of a finitely connected domain is
defined to be the number of the components of its complement.

\begin{thm} \label{t:main}
Let $K$ be a compact subset and let  $\mu$ be a finite positive
measure supported on $K$.  If each of the components of $K^{\circ}$
is multi-nicely connected and the harmonic measures of the
components of $K^{\circ}$ are mutually singular, then there exists a
$Borel$ partition $\{ \Delta_{n}\}_{n=0}^{\infty} $ of $supp(\mu)$
such that
 \[ A^{q}(K,\mu) = L^{q}(\mu | \Delta_{0})  \oplus A^{q}(K, \mu | \Delta_{1})
\oplus ... \oplus A^{q}(K, \mu | \Delta_{n}) \oplus ...
\]
and for each $n \geq 1$, if $U_{n} $ denotes $ \nabla A^{q}(K,
\mu|\Delta_{n})$,  then

1) $\overline{U}_{n} \supset  \Delta_{n} $ and
 $A^{q}(K, \mu | \Delta_{n}) = A^{q}(\overline{U}_{n}, \mu | \Delta_{n})$;

2) each  $U_{n}$ is a finitely connected domain that is conformally
equivalent to a circular domain $W_{n}$; the connectivity of $U_{n}$
does not exceed the connectivity of the component of $K^{\circ}$
that contains $U_{n}$;

3) the map $e$, defined by $e(f)=\hat{f}$, is an isometrical
isomorphism and a weak-star homeomorphism from
 $A^{q}(K_{n}, \mu|\Delta_{n}) \bigcap L^{\infty}(\mu|\Delta_{n})$
onto $H^{\infty}(U_{n})$;

4) $\mu| \partial U_{n} \ll \omega_{U_{n}}$; 
and if $u_{n}$ is a conformal map from  $W_{n}$ onto $U_{n}$, then
for each $f\in H^{\infty}(U_{n})$ has nontangential limits on
$\partial U_{n}$ {\em a.e.} $[\mu] $ and
\[ e^{-1}(f)(a) = \lim_{z\rightarrow u_{n}^{-1}(a)} f\circ u_{n}(z) \hspace{.05in}
\mbox{\em{a.e.} on} \hspace{.05in}\partial U_{n}
\hspace{.05in}\mbox{ with respect to }\hspace{.05in}  \mu|\partial
U_{n};
\]

5)         for  each $f\in H^{\infty}(U_{n})$, if let $f^{*}$  be
equal to its nontangential limit values on $\partial U_{n}$  and let
$f^{*} = \hat{f}$ on $U_{n}$, then the map $m$, defined by $m(f) =
f^{*}|\Delta_{n}$, is the inverse of the map $e$.

\end{thm}

\begin{rem}
Thomson proved 1), 2) and 3) of Theorem~\ref{t:main} in the case
when $A^{q}(K,\mu)=P^{q}(\mu)$.  For the polynomial case, 4) is the
main result in \cite{oy}.  The author proved 4)  for $A^{q}(K,\mu)$
with a different method in \cite{comm}.
\end{rem}

\begin{rem}
3) clearly implies that each $A^{q}(K,\mu_{n})$ is irreducible.
\end{rem}

\begin{rem}
The condition on $K$ is the best possible one.  What we mean here is
that in order to have Theorem~\ref{t:main} holds for any positive
finite measure supported on $K$, it is necessary and sufficient that
each component of $K^{\circ}$ is multi-nicely connected and the
harmonic measures of the components of $K^{\circ}$ are mutually
singular.
\end{rem}

Now we outline a proof for this fact. Let $\Omega$ be a component of
$K^{\circ}$ and let $\mu$ be a harmonic measure for $\Omega$. By
Theorem 3 of \cite{comm}, the map $f\rightarrow \hat{f}$, from
 $A^{q}(\overline{\Omega}, \mu) \bigcap L^{\infty}(\mu)$ onto
$H^{\infty}(\Omega)$ is surjective if and only if  $\Omega$ is a
multi-nicely connected domain. Hence we know that $\Omega$ must be
multi-nicely connected.

Now let $\Omega_{1}$ and $\Omega_{2}$ be two components of
$K^{\circ}$. We want to show that $\omega_{\Omega_{1}}$ and
$\omega_{\Omega_{2}}$ are mutually singular.  Set
$\mu=\omega_{\Omega_{1}} +\omega_{\Omega_{2}}$. Then it is easy to
see that $A^{q}(K,\mu)$ is pure. Since the harmonic measure at a
given point is a representing measure, it follows by the definition
of $abpe$ and the Harnack's inequality that $ \nabla A^{q}(K,\mu)
\supset \Omega_{1} \cup \Omega_{2}. $ Since each $\Omega_{i}$ has no
boundary slit, it follows  clearly that $\nabla A^{q}(K,\mu) =
\Omega_{1} \cup \Omega_{2}.$ If Theorem~\ref{t:main}  holds for
$A^{q}(K,\mu)$, then
\[ A^{q}(K,\mu) = A^{q}(\overline{\Omega}_{1},\mu_{1})  \oplus
A^{q}(\overline{\Omega}_{2},\mu_{2}),
\]
where $\mu_{i}$, $i=1, 2$, are as in Theorem~\ref{t:main}. Let $v$
be a conformal map of $\Omega_{1}$ onto a circular domain $W$.
Theorem~\ref{t:main} implies that there exists $v_{e}\in
A^{q}(\overline{\Omega}_{1}, \mu_{1})$ such that $\hat{v}_{e} = v$.
Set $\eta = \mu_{1} \circ v_{e}^{-1}$. According to  Lemma 2 in
\cite{comm}, $\eta$ is a measure on $\partial W$ such
 that $A^{q}(\overline {W},\eta)$ is irreducible and
$\nabla A^{q}(\overline{W},\eta) = W$. Moreover, $\eta \ll
\omega_{W}$ by Lemma 3 of \cite{comm}. On the other hand, since $W$
is circular and since $\nabla A^{q}(\overline{W},\eta) = W$, it is
easy to see that $A^{\infty}(\overline{W}, \eta)$, the weak-star
closure of $A(\overline{W})$ in $L^{\infty}(\eta)$, is equal to
$\widetilde H^{\infty}(W)$, which is the image of the map
$f\rightarrow \tilde f$ from $H^{\infty}(W)$ into
$L^{q}(\omega_{W})$ (where $\tilde f$ is the boundary value function
of $f$ on $\partial W$). Since the support of
$\mu_{1}\subset 
\partial \Omega_{1}$,  it follows by a classical result that
[$\omega_{W}] = [\eta]$.  Now, applying Lemma 3 in \cite{comm}, we
conclude that $[\omega_{\Omega_{1}}]=[\mu_{1}]$. Similarly, we have
$[\mu_{2}]=[\omega_{\Omega_{2}}]$. But $\mu_{1}$ and $\mu_{2}$ are
mutually singular, therefore $ \omega_{\Omega_{1}}$ and $
\omega_{\Omega_{2}}$ must be mutually singular.

\section{The proof  of the main result}
\begin{lem}
For each $f \in A^{q}(K,\mu),\hspace{.03in} \hat{f}=f$ on $\nabla
A^{q}(K,\mu)$ a.e. $[\mu]$.
\end{lem}

\begin{proof}
Let $ a$ be an $abpe$ and  choose a sequence $\{f_{n}\} $ in $A(K)$
such that $f_{n}\rightarrow f $ in $L^{q}(\mu)$. Since
$f_{n}\rightarrow \hat{f}$ uniformly in a neighborhood of $a$, it
follows (by passing a sequence if necessary) that $f_{n}\rightarrow
f$ {\em a.e.} $[\mu]$ and consequently
\[ \hat{f}(a) = \lim_{n\rightarrow \infty} \hat{f}_{n}
= \lim_{n\rightarrow \infty} f_{n} = f(a) \hspace{.07in}a.e.
\hspace{.07in} [\mu].
\]

\end{proof}

The following lemma is elementary too.

\begin{lem} \label{l:elem1}
If $f\in L^{\infty}(\mu) \cap A^{q}(K,\mu)$ and $g\in A^{q}(K,\mu)$,
then $fg \in A^{q}(K,\mu)$ and $\widehat{fg} = \hat{f}\hat{g}$.
\end{lem}

The next lemma is proved in \cite{comm}.

\begin{lem} \label{l:ISO}
Let $\Omega = \nabla A^{q}(K,\mu)$. If $\Omega$ is finitely
connected, then every component of ({\bf C} $\setminus \Omega$) has
nonempty interior.
\end{lem}

\vspace{.05in} \noindent{\em Representing measures}. Let $\cal B$
be a closed subalgebra of $C(K)$. A complex representing measure of
$\cal B$ for $a\in K$
 is a finite measure $\nu$ on $K$  such that
\[ f(a) = \int f d\nu, \hspace{.05in} f\in \cal B.
\]
A {\em representing measure} for $a$  is a probability measure that
satisfies the above condition.  Note, if $a$ is a  peak point, then
the only representing measure for $a$ is the point mass
$\delta_{a}$.

\vspace{.05in} \noindent{\em The sweep of a measure}. Let $G$ be a
domain that is regular for the Dirichlet problem and let $\mu$ be a
measure on $\overline{G}$. The sweep of $\mu$ is the unique positive
measure $\tilde{\mu}$ on $\partial G$ that satisfies  $\int
_{\overline {G}}\tilde{u} \hspace{.05in}d\mu = \int_{\partial G}
 u d\tilde{\mu}, \hspace{.05in} u \in C(\partial G)$,
where $\tilde{u}$ is the solution of the Dirichlet problem for $u$.
 A simple fact is that if $\mu$ is a measure on $\overline{G}$,
then $\tilde{\mu} = \mu|\partial G + \widetilde{\mu|G}$.

\begin{lem}  \label{ABS1}
Let $K$ be a compact subset such that the components of $K^{\circ}$
are multi-nicely connected and the harmonic measures of the
components of $K^{\circ}$ are mutually singular. Let $\Omega$ be a
component of $K^{\circ}$.  If $A^{q}(K,\mu)$ is pure and if
$K^{\circ}$ is dense in $K$, then $\mu|\partial \Omega
\ll\omega_{\Omega}$.
\end{lem}

\begin{proof}
Let $\{\Omega_{j}\}_{j=0}^{\infty}$ be the collection of the
components of $K^{\circ}$.  Fix an integer $j\geq 0$ and let $E$ be
a component
 of $\partial \Omega_{j}$.  Let $G_{j}$ be the unique simply connected
domain in the sphere {\bf C}$_{\infty}$ that has
 $E$ as its boundary and contains $\Omega_{j}$.
Since $\Omega_{j}$ is multi-nicely connected, $G_{j}$ must be nicely
connected. For $i \neq j$, let $G_{i}$ be the bounded simply
connected domain that contains $\Omega_{i}$ and whose boundary is a
component of $\partial \Omega_{i}$. Clearly, $G_{i}$ is also nicely
connected. Now let $\Omega$ be the union of those $G_{i}$'s for
which $G_{i} \cap G_{j} = \emptyset$ (different $G_{i}$'s are either
disjoint or one contains other).  Set $ G= \Omega \bigcup  G_{j}$.
Then each component of $G$ is equal to some $G_{i}$. Let
$\{G_{i_{k}}\}$ be the collection of all the components of $G$. Then
our hypothesis on $K$ implies that the harmonic measures of the
components of $G$ are mutually singular. It follows from \cite{davi}
that $A(G)$ is a Dirichlet algebra on $\partial G$. Hence, every
point in $\partial G$ is a peak point for $A(G)$ and every trivial
Gleason part of $A(G)$ consists of a single point. Therefore,
$\{G_{i_{k}}\}$ is the collection of all the nontrivial Gleason
parts of $A(G)$.

Let $\eta \perp A(G)$. By the {\em Abstract F. and M. Riesz Theorem}
\cite{ga}, $\eta = \sum_{m\geq 0} \eta_{m}$, where each $\eta_{m}
\perp A(G)$,
$\eta_{m} \ll v_{m}$  for a representing measure $ v_{m}$  at some
point $a_{m}$ in $\overline{G}$, $v_{m}$'s are mutually singular.
Let $a \in \partial G$.  Then $a$ is  a peak point. Let $f\in A(G)$
be a peak function for $a$. Then $f^{n}(z) \rightarrow \chi_{\{a\}}$
pointwise, and thus $0 = \int \lim_{n\rightarrow \infty} f^{n} d
\eta_{m}
 = \eta_{m}(\{a\})$.
Hence $ a_{m} \in G$ (otherwise, $\nu_{m}$ is the point mass at
$a_{m}$ and hence $\nu_{m} (\overline{G} - \{a_{m}\}) =0$. So we
conclude that $\eta_{m}(\overline{G}) = \eta_{m}(\overline{G} -
\{a_{m}\}) + \eta_{m}(\{a_{m}\}) = 0+0=0$, a contradiction).

Let $G_{i_{k_{m}}}$ be the component that contains $a_{m}$ and let
$\tilde{v}_{m}$ be the sweep of $v_{m}$ on $\partial G_{i_{k_{m}}}$.
Then  for each $ g \in A(G)$

\begin{align*}
\int_{\partial G} g(z) d \tilde{v}_{m}   =
 \int_{\partial G_{i_{k_{m}}}} g(z) d \tilde{v}_{m}
= \int_{\overline{G}_{i_{k_{m}}}} g(z) d v_{m}  
= \int_{\partial G_{i_{k_{m}}}} g(z) d \omega_{a_{m}},
\end{align*}
where $\omega_{a_{m}}$ is the harmonic measure of $G_{i_{k_{m}}}$
evaluated at $a_{m}$. It follows by the uniqueness that $
\tilde{v}_{m}= \omega_{a_{m}}$. Hence, we have
\begin{align*}
 \eta|\partial G   \ll (\sum v_{a_{m}})|\partial G
\ll \sum \tilde{v}_{a_{m}}|\partial G
 = \sum \omega_{a_{m}}|\partial G
\end{align*}
In particular, we have that $\eta|E \ll \omega_{\Omega_{j}} |E $.

Finally, suppose that $g\in L^{p}(\mu)$ such that $\int f g d \mu =
0$, for $f \in A(K)$, where $\frac{1}{q} + \frac{1}{p} = 1$. Then $g
\perp A(G)$ as well. Hence, $g\mu |E \ll \omega_{\Omega_{j}}|E$.
This implies that $(g\mu)_{s}|E = 0$, where $(g\mu)_{s}$ is the
singular part of the {\em Lebesgue decomposition} of $g\mu$ with
respect to $\omega_{\Omega_{j}}|E$. Consequently,  we have $g\perp
\chi_{\Delta\cap E}$, where $\Delta$ is the carrier of
$\mu_{s}|\partial \Omega_{j}$ and $\mu_{s}$ is the singular part of
the  {\em Lebesgue decomposition } of $\mu$ with respect to
$\omega_{\Omega_{j}}$. Now an application of {\em Hahn-Banach
theorem} yields $\chi_{\Delta\cap E} \in A^{q}(K,\mu)$. By purity,
$\chi_{\Delta\cap E} = 0$ a.e. $[\mu]$ and therefore $\mu|E \ll
\omega_{\Omega_{j}} |E $. Since $E$ is an arbitrary component of
$\partial \Omega_{j}$, it follows that $\mu|\partial \Omega_{j} \ll
\omega_{\Omega_{j} }$.

\end{proof}

\begin{prop} \label{p:ABS2}
Let $K$ be a compact subset such that the components of $K^{\circ}$
are multi-nicely connected and the harmonic measures of $K^{\circ}$
are mutually singular.  Let $\{\Omega_{j}\}_{j=0}^{\infty}$ denote
the collection of the components of $K^{\circ}$. Set $\omega =
\sum_{j=0}^{\infty} \frac{1}{2^{j}} \omega_{\Omega_{j}}$. If
$A^{q}(K,\mu)$ is pure, then $\mu|\partial K \ll\omega$.
\end{prop}
\begin{proof}
By Lemma 17.10 in \cite[p. 246]{conw}, there exists a function
 $g\in A^{q}(K,\mu)^{\perp}$ such that $|f|\mu \ll |g|\mu$ for each
$f\in A^{q}(K,\mu)^{\perp}$. Since $A^{q}(K,\mu)$ is pure, we see
that $g \neq 0$ on $supp(\mu)$ {\em a.e.} $[\mu]$.
 Thus $[|g|\mu] = [\mu]$ {\em a.e.} $[\mu]$. Set $\nu = |g|\mu$.

Let  $a\in K -\overline{K^{\circ}}$ be such that
 $\int \frac{d |\nu|} {|z-a|} < \infty$  and let $f$ be a peak
function for $a$. For each integer $n\geq 1$,
 $\frac{1-f^{n}(z)}{z-a} \in A(K)$, so we have
 \[\hat{\nu}(a) = \lim_{n\rightarrow \infty}
\int \frac{1-f^{n}(z)}{z-a} d\nu = 0
\]
Since the set $\{a: \int \frac{d |\nu|} {|z-a|} < \infty\}$ has full
area measure in the plane, it follows by a well-known fact (see
Theorem~\ref{t:light} or the comments after it) that $\nu$ is the
zero measure  off $\overline{K^{\circ}}$.

Finally, since $\mu|\partial \Omega_{j} \ll \omega_{\Omega_{j}}$ for
each $j$ (by Lemma~\ref{ABS1}), the conclusion follows.

\end{proof}

The proof of the next lemma can be found in \cite{b_apr}.

\begin{lem} \label{l:key}
Let $\Omega$ be an open subset whose boundary does not contain
$\infty$. Suppose the components of $\Omega$ are  multi-nicely
connected and the harmonic measures of the components of $\Omega$
are mutually singular. Let $\{\Omega_{j}\}_{j=0}^{\infty}$ be the
collection of the components of $\Omega$. If all but finitely many
components of $\Omega$ are simply connected, then $A(\Omega)$ is
boundedly pointwise dense in $H^{\infty}(\Omega)$.
\end{lem}

For a finite positive  measure $\mu$, let $A^{\infty}(K,\mu)$ denote
the weak-star closure of $A(K)$ in $L^{\infty}(\mu)$.

\begin{lem}
Let $\Omega$ be a multi-nicely connected domain. Let $W$ be a
circular domain that is conformally equivalent to $\Omega$ and let
$\phi$ be a conformal map of \hspace{.02in} $W$ onto $\Omega$. Then
the boundary value function $\tilde{\phi}$ has a well-defined
inverse, $\tilde{\phi}^{-1}$, on $\partial \Omega$. Moreover,
 $\tilde{\phi}^{-1}\in A^{\infty}
(\overline{\Omega}, \omega_{\Omega})$.

\end{lem}

\begin{proof}
Since $\phi$ is almost 1-1 on $\partial W$ with respect to
$\omega_{W}$, it is apparent that $\tilde{\phi}$ has a well-defined
inverse function on a set of full $\omega_{\Omega}$ measure. By
Lemma~\ref{l:key} $A(\Omega)$ is boundedly pointwise dense in
$H^{\infty}(\Omega)$. Choose a bounded sequence $\{f_{n}\}$ in
$A(\Omega)$ such that $f_{n}\rightarrow \phi^{-1}$ on $\Omega$. Then
one can show that $f_{n} \rightarrow \tilde{\phi}^{-1}$ in the
weak-star  topology of $L^{\infty}(\omega_{\Omega})$. Hence
$\tilde{\phi}^{-1} \in A^{\infty}(\overline{W}, \omega_{\Omega})$.

\end{proof}
The following is one of our key lemmas.

\begin{lem} \label {l:key2}
Suppose that each component of $K^{\circ}$ is multi-nicely connected
and the harmonic measures of the components of $K^{\circ}$ are
mutually singular. Let $U$ be a component of $\nabla A^{q}(K,\mu)$
and let $\Omega$ be the component of $K^{\circ}$ that contains $U$.
Set $\tau = \mu|\overline{\Omega}$. Then $
A^{q}(\overline{\Omega},\tau)\subset A^{q}(K,\mu)$ and $ U \subset
\nabla A^{q}(\overline{\Omega},\tau)$.
\end{lem}

\begin{proof}
We first assume that $A^{q}(K,\mu)$ is pure. Let
$\{\Omega_{i}\}_{i=0}^{\infty}$ be the collection of all the
components of $K^{\circ}$. Without loss of generality, let
$\Omega_{0} = \Omega$. Suppose $h\in A^{q}(\overline{\Omega}, \tau)$
and choose a sequence $\{r_{n}\}$ in $A(\overline{\Omega})$ such
that $r_{n}\rightarrow h $ in $L^{q}(\tau)$. Let $\omega = \sum_{i}
\frac{1}{2^{i}} \omega_{\Omega_{i}}$. Fix a function $r_{n}$. Extend
both $h$ and $r_{n}$ to be  functions on the whole
 plane by defining their  values to be
zero off $\overline{\Omega}$. We claim that  there exists a sequence
 $\{q_{n}\}$ in $A(K^{\circ})$  such that it weak-star converges to $r_{n}$ in
$L^{\infty}(\omega)$.

For each $i\geq 1$, let $G_{i}$ be the bounded simply connected
domain that contains $\Omega_{i}$ and whose boundary is a component
of $\partial \Omega_{i}$. Let $G$ be the union of $\Omega$ with
those domains $G_{i}$ that do not intersect $\Omega$. Then all but
finitely many components of $G$ are simply connected domains and
each component of $G$ is multi-nicely connected.
 By Lemma~\ref{l:key},
$A(G)$ is boundedly pointwise dense in $H^{\infty}(G)$. Thus, there
exists a bounded sequence $\{q_{m}\}$ in $A(G)$ so that it pointwise
converges to $r_{n}$ on $G$. Now for given $\epsilon > 0$, let
$f\in L^{1}(\omega)$.  Then
\[ |\int_{\partial K} f(r_{n}-q_{m}) d\omega | \leq
 |\int_{\cup_{o}^{k} \partial \Omega_{i}} f(r_{n}-q_{m}) d\omega |
+ \frac{\epsilon}{2} \hspace{.05in} \mbox{for all} \hspace{.05in} m
\]
whenever $k$ is sufficiently  large.
 Observe that $\{q_{m}\}$ weak-star converges to $r_{n}$ in
$L^{\infty}(\omega_{\Omega_{i}})$ for each $i$.  Thus,
\[ |\int_{\partial K} f(r_{n}-q_{m}) d\omega | \leq \frac{\epsilon}{2}
+ \frac{\epsilon}{2} \hspace{.05in}\mbox{when}
 \hspace{.05in}m \hspace{.05in}\mbox{is sufficiently large.}
\]
Hence, for each $ f\in L^{1}(\omega) $ we have
\[ \lim_{n\rightarrow \infty} \int_{\partial K} f(r_{n}-q_{m}) d\omega
= 0.
\]
That is, $\{q_{m}\}$ weak-star converges to $r_{n}$ in
$L^{\infty}(\omega)$. This proves the claim.

Next we show that $r_{n}$ belongs to the weak-star closure
 of $A(\overline{K^{\circ}})$ in $L^{\infty}(\mu)$.
By Proposition~\ref{p:ABS2}, $\mu|\partial K \ll \omega$. Thus
\[ \lim_{m\rightarrow} \int_{\partial K} f q_{m} \hspace{.05in}d\mu = \int_{\partial K}
f r_{n} d \mu, \hspace{.05in}f \in L^{1}(\mu).
\]
Since $\{q_{m}\}$ is bounded and pointwise converges to $r_{n}$ on
$K^{\circ}$, it follows by the bounded convergence  theorem that
\[ \lim_{n\rightarrow } \int fq_{m} \hspace{.05in}d\mu = \int f r_{n} \hspace{.05in}d\mu, \hspace{.05in}
f \in L^{1}(\mu).
\]
Therefore, $r_{n}$ belongs the weak-star closure of
$A(\overline{K^{\circ}})$ in $L^{\infty}(\mu)$. But this closure is
contained in the weak closure of
 $A(\overline{K^{\circ}})$ in $L^{q}(\mu)$.
Because a convex set is norm-closed if and only if it is weakly
closed in $L^{q}(\mu)$, we have $r_{n}$ is in
$A^{q}(\overline{K^{\circ}},\mu)$.

Now for each $n \geq 1$,  choose $x_{n} $ in
 $A(\overline{K^{\circ}})$ such that
\[ \|r_{n}- x_{n} \|_{A^{q}(\overline{K^{\circ}},\mu)}| \leq \frac{1}{n}.
\]
Then
\begin{align*}
\|x_{n}-h\|_{L^{q}(\mu)} & = \|x_{n}-r_{n}\|_{L^{q}(\mu)} +
 \|r_{n}-h\|_{L^{q}(\mu)}  \\ &
\leq \frac{1}{n} + \|r_{n}-h\|_{L^{q}(\tau)} \rightarrow 0,
\hspace{.07in} \mbox{as}\hspace{.05in}\rightarrow \infty.
\end{align*}
  Thus we have that
$h\in A^{q}(\overline{K^{\circ}},\mu)$. Because  $h$ is an arbitrary
element in $A(\overline{\Omega},\tau)$, we conclude that
$A^{q}(\overline{\Omega},\tau) \subset
A^{q}(\overline{K^{\circ}},\mu).$ Since $A^{q}(K,\mu)$ is pure and
since $\mu$ is supported on $\overline{K^{\circ}}$ (for
$\mu|\partial K \ll \omega$ and $\omega$ is supported on $ \partial
K^{\circ}$),
we have that
 \[A^{q}(K,\mu) = A^{q}(\overline{K^{\circ}},\mu).\]
 Hence
$A^{q}(\overline{\Omega},\tau) \subset A^{q}(K,\mu)$.

Now let $b \in U$. Then $b$ is an $abpe$ for $A^{q}(K,\mu)$ and thus
there exists $d>0$ and a small open disk
 $D_{b} \subset U$
such that for all $r\in A(K)$
\[ |r(a)| \leq d\{\int |r|^{q} d \mu \}^{\frac{1}{q}}, \hspace{.05in}
a \in D_{b}.
\]
Let $y \in A(\overline{\Omega})$ and extend $y$ to be zero off
$\overline{\Omega}$. Then $y \in A^{q}(K,\mu)$ and so there is a
sequence $\{q_{n}\}$ in $A(K)$ so that it converges to $y$ in
$L^{q}(\mu)$. Then $\{q_{n}\}$ converges to $y$ uniformly on
$D_{b}$. Hence, it follows by the expression above that for all
$y\in A(\overline{\Omega})$
\[ |y(a)| \leq d \{ \int |y|^{q} d\tau \}^{\frac{1}{q}},\hspace{.05in}
a\in D_{b}.
\]
Thus $ a\in \nabla A^{q}(\overline{\Omega},\tau)$. By the definition
of $abpe$, $U\subset  \nabla A^{q}(\overline{\Omega},\tau)$.

If $A^{q}(K,\mu)$ is not pure, let $\mu=\mu_{0}+\mu_{1}$ be the
decomposition so that $A^{q}(K,\mu) = L^{q}(\mu_{0}) \oplus
A^{q}(K,\mu_{1})$ and $A^{q}(K,\mu_{1})$ is pure. Then
\[ \nabla A^{q}(K,\mu) \supset \nabla A^{q}(K,\mu_{1}) \supset U.
\]
So the conclusion of the lemma follows.

\end{proof}

A function $f$ analytic at $\infty$ can be written as a power series
of the local coordinate $\frac{1}{z-z_{0}}$ at $\infty$:
\[ f(z)= f(\infty)+\frac{a_{1}}{z-z_{0}}+\frac{a_{2}}{(z-z_{0})^{2}}
+ ...  .\] The coefficient $a_{1}$ is called the derivative of $f$
at $\infty$ and is denoted by $f^{'}(\infty)$. It is easy to see
that $f^{'}(\infty) = \lim_{z\rightarrow \infty} z(f(z)-f(\infty))$.
Define $\beta(f,z_{0}) = a_{2}$.

The next lemma is elementary.

\begin{lem} \label{l:elem}
Let $\delta > 0$ and let $a\ in $ {\bf C}. Let $B(a,\delta)= \{z:
|z-a| \leq \delta \}$. If $f$ is a bounded analytic function on {\bf
C}$_{\infty}\setminus B(a,\delta)$, then $ |f^{'}(\infty)| \leq
\delta \|f\|_{\infty}$ and $|\beta(f,a)| \leq
\delta^{2}\|f\|_{\infty}$.
\end{lem}

\vspace{.05in} \noindent{\em Thomson's Scheme.}~\footnote{In this
paper, we don't directly use this scheme. But we need the concept of
light and heavy points and results related to them
(Lemma~\ref{l:cru} and Theorem~\ref{t:light}).} Now,  we introduce
an approximation scheme originally developed by J.
Thomson in \cite{jm}.

For an integer $k \geq 1$, let $\{S_{kp} \}_{p=1}^{\infty}$ be the
collection of all open squares with sides $2^{-k}$, parallel to the
coordinate axes and corners at the points whose coordinates are both
integral multiples of $2^{-k}$. A finite sequence $\{S_{i}
\}_{i=1}^{n}$ of squares is called a path of squares if the interior
of $\cup \overline{S}_{i}$ is connected. In this case we say $S_{1}$
and $S_{n}$ are joined by a path of squares.  The collection of
$\{S_{kp} \}_{p=1}^{\infty}$ is called the $k$-th generation of
squares.

Let $\phi$ be a nonnegative function in $L^{1}($\mbox{Area}$)$. An
open square $S$ is said to be {\em light} with respect to $\phi$ if
$\int_{S}\phi \hspace{.05in} d \hspace{.02in}\mbox{Area} \leq
\mbox{[Area}(S)]^{2}$.

Now we begin with the scheme. Let $a \in $ {\bf C} and let $S$ be a
square in   $\{S_{kp} \}_{p=1}^{\infty}$ such that $a \in \overline
S$. Color $S$ yellow and let $\Gamma_{k} =  \partial S$. We then
move to the  squares in the next generation. First, color green
every light square in $\{S_{(k+1)p} \}_{p=1}^{\infty}$  that lies
outside $\Gamma_{k}$ and has a side on $\Gamma_{k}$. Second, color
green every light square that can be joined  to a green square in
the first step by a path of light squares in $\{S_{(k+1)p}
\}_{p=1}^{\infty}$. Now if there is an unbounded green path (that is
made up by infinitely many squares), then this coloring process
ends. Otherwise, let $\gamma_{k+1}$ be the boundary of the
polynomially convex
 hull~\footnote{
The polynomially convex hull of a compact subset $K$ in the plane is
defined as the union of $K$ and all the bounded components of {\bf
C}$\setminus K$.} of the union of $\Gamma_{k}$ and the closure of
the green squares.  We then color red every square $S$ in
$\{S_{(k+1)p} \}_{p=1}^{\infty}$ if $S$ is outside $\gamma_{k+1}$
and $S$ has a side on $\gamma_{k+1}$. After that, color a square $T$
yellow if $T$ is outside $\gamma_{k+1}$ and $T$ has
 no side lying on $\gamma_{k+1}$
and the distance from $T$ to some red square in $\{S_{(k+1)p}
\}_{p=1}^{\infty}$  is less or equal to $(k+1)^{2} 2^{-(k+1)}$. Now
let $\Gamma_{k+1}$ be the boundary of the polynomially convex hull
of the union and the closure of the colored squares in the
$(k+1)$-th generation. To this step, the coloring process in
$(k+1)$-th generation of squares is completed.

Next we continue this process to the $(k+2)$-th generation of
squares and keep this process to all higher generations unless there
is an unbounded green path in the coloring scheme in some $(m+l)$-th
generation $(l \geq 1)$. We use $(\phi, a, k)$ to denote this
colored scheme.

\vspace{.05in} \noindent {\em Light and heavy points}. For a
nonnegative function $\phi \in L^{1}($\mbox{Area}$)$, a point
$\lambda $ in {\bf C}  is called light (with respect to $\phi$) if
there exists $\delta > 0$ such that  for each
 $\delta_{0} \leq \delta$ \[ \{z: |z-a| = \delta_{0} \} \cap
\{\mbox{all colored squares in}\hspace{.05in} (\phi,k,a)\} \neq
\emptyset,
\]
whenever $k$ is a sufficiently large integer. If a point is not
light, then it is called a heavy point.

\begin{rem}
The construction of our colored scheme is exactly the same as that
in \cite{jm}.  But the light and heavy points improved '{\em light
route to $\infty$'} and '{\em heavy barrier}' in Thomson's original
work.  Let us explain the difference: For a given $\phi$, if there
is an unbounded green path
 in the colored scheme $(\phi,k,a)$ for every $k$, it is said that there
is a sequence of light routes from $a$ to $\infty$. This is
essentially the definition of 'light' points in Thomson's paper.
Because most of the light points  for a given $\phi$ in our
definition don't have a sequence of light routes from $a$ to
$\infty$, the set of the light points is much larger than the set of
points that have a sequence of light routes from $\infty$.
\end{rem}

\vspace{.05in} The Cauchy transform of a measure (with compact
support) $\mu$ is defined as $\hat{\mu}(z) = \int \frac{
d\mu(w)}{w-z}$. Because $\frac{1}{z}$ is local integrable with
respect to the  $\mbox{area}$ measure, it  follows that
$\hat{\mu}(z)$ is defined everywhere except a subset of zero area.

\vspace{.05in}

The following is a practically useful result coming out of our light
point concept  \cite[Theorem 2.4]{qiu1}.

\begin{thm} \label{t:light}
Let $\mu$ be a finite measure with compact support. Let $V$ be an
open subset in {\bf C}.  If  every point in $V$ is light with
respect to
 $|\hat{\mu}|$, then $|\mu|(V) = 0$.
\end{thm}

The above theorem  generalizes a well-known result in the theory of
uniform approximation: {\em if $\hat{\mu} = 0$ {\em a.e.} on an open
subset with respect to the area measure, then the restriction of
$\mu$ on the open subset is zero.} The following is the key lemma
 in \cite[Lemma 2.3]{qiu1}.

\begin{lem} \label{l:cru}
Let $\nu_{1}, ..., \nu_{k}$ be finite measures. Let $\phi(z) =
\max\{|\hat{\nu}_{i}(z)|: 1\leq j\leq k \}$.  If $ a$ is a  light
point with respect to $\phi$, then there is  an arbitrarily small
positive number $\delta$ such that for any $\epsilon > 0$ and
$\alpha,\beta \in \{z: |z-1|\leq 1\} $, there is a function in
$C(${\bf C}$_{\infty}$) that has the following properties: 1)
$\|f\|_{\infty} \leq C$ (a universal constant), 2) $f$ is analytic
on $\{z: |z-a| > \delta \}$, 3) $f(\infty) = 0$, 4) $f^{'}(\infty) =
\alpha \delta$, 5) $\beta(f,a) = \beta \delta^{2}$, 6) $ |\int f
d\nu_{j}|\leq \epsilon$ for all $1\leq j\leq k$.
\end{lem}

\vspace{.05in} \noindent {\em Vitushkin covering}. For a natural
number $k$, let $\{S_{kl} \}_{l=1}^{\infty}$ is the $k$-th
generation of squares with sides of length $2^{-k}$. For each
$S_{kl}$, let $F_{kl}$ be the square obtained by enlarging $S_{kl}$
$\frac{5}{4}$ times. The collection $\{F_{kl} \} $ is called a
regular Vitushkin covering of the plane. We suppress $k$ and let
$z_{l}$ be the center  of $F_{l}$. Then there exists a $C^{1}$
partition of unity $\{\phi_{l}\}$ subordinate to $\{F_{l}\}$ with
$\|\mbox{grad} \phi_{l}\| \leq 100\hspace{.05in} 2^{k}$ such that
\[ \sum_{l}^{\infty} \min(1, \frac{2^{-3k}}{|z-z_{l}|^{3}}) \leq
C \min\{1, \frac{2^{-k}}{dist(z, \cup_{l} F_{l})} \},
\hspace{.3in}z\in \mbox{\bf C}.
\]
One may consult \cite{conw} for a proof of the inequality.
\begin{lem} \label{l:scheme}
Suppose that each component of $K^{\circ}$ is multi-nicely connected
and the harmonic measures of the components are mutually singular.
Let $U$ be a component of $\nabla A^{q}(K,\mu)$ and let $f \in
H^{\infty}(U)$. Then there exists a function $h \in A^{q}(K,\mu)\cap
L^{\infty}(\mu)$ such that $\hat{h}(z) = f(z)$ on $U$ and $h = 0$
off\hspace{.05in} $\overline{U}$.

\end{lem}

\begin{proof}
First, we assume that $K$ is finitely connected, $K=\overline
K^{\circ}$ and $K^{\circ}$ is connected. Let $\{x_{j}\}$ be a
countable dense subset of $A^{q}(K,\mu)^{\perp}$. For an integer
$k\geq 1$, let $\Phi(z)= \mbox{max}\{|(\widehat{x_{j}\mu})(z)|:
j\leq k\}$.
 Let $\{F_{l}\}$ be the regular Vitushkin
covering of squares with sides of length $\frac{5}{4} 2^{-k}$ and
center $z_{l}$.  Then there is a
 $C^{1}$ partition $\{\phi_{l}\}$ subordinate to the
covering $\{F_{l}\}$. For each $l$, let
 $f_{l}= T_{\phi_{l}} f = \frac{1}{\pi} \int
 \int \frac{f(z) - f(w)}{z-w} \frac{\partial
\phi_{l}}{\partial \overline{z}} \hspace{.05in}d\mbox{ Area}. $ Then
$f_{l}$ is analytic off $F_{l}$, $f_{l}(\infty)=0$,
 and
$\|f_{l}\|_{\infty} \leq 2\|\mbox{grad} \phi_{l}\|\mbox{diam}
[supp(\phi_{l}] \sup\{|f(z)-f(w)|: z,w \in supp(\phi_{l}\} \leq
C_{0}, $ where $C_{0}$ is a  positive universal constant.

Let $l$ be such that $F_{l} \cap \partial U \neq \emptyset$
 and let $a \in \partial U\cap F_{l}$.  We claim that $a$ is a light
point with respect to $\Phi$. In fact, first  let $a\in K^{\circ}$.
Set $V=K^{\circ}\setminus \partial U$,
it follows by Lemma 3.3 in \cite{qiu1} that 
$|(\widehat{x_{j}\mu})(z)|=0$ on $V$ for each $j$ and thus $\Phi=0$
on $V$. By Lemma 3.7 in \cite{qiu1}  we see $a$ is light.\footnote{
By combining the proof Lemma 3.3 in \cite{qiu1} and that of Theorem
4.8 in \cite{jm}, we can extend Lemma 3.3 in \cite{qiu1} so that it
has the conclusion that Theorem 4.8 in \cite{jm} has (so it can be
applied to $\sum_{j\leq k} |x_{j}\mu|$)). From this, we also see any
point in $K^{\circ}$ that is not  an $abpe$ must be light.}
Now suppose  that $a \in \partial K$. Since $\nu_{j} \perp A(K)
\supset R(K)$, we have  that  $\hat{\nu}_{j} = 0$ off $K$. Thus,
again  it follows from  Lemma 3.7 in \cite{qiu1} that $a$ is also a
light point. This proves the claim.

Next let $d_{l} = \frac{1}{2} 2^{-k}$ and let  $B(a, d_{l})$ be open
disk having radius $d_{l}$ and the center at  $a$. Applying
Lemma~\ref{l:elem} to $B(a, d_{l})$,  it follows that
$|f_{l}^{'}(\infty)| \leq C_{0} d_{l} \hspace{.05in} \mbox{and}
\hspace{.05in} \beta(f_{l},z_{l}) \leq C_{0} d_{l}^{2}. $ Let
$\alpha=\frac{f_{l}^{'}(\infty)}{C_{0}d_{l}}$ and $\beta=
\frac{\beta(f_{l},z_{l}) }{C_{0}d_{l}^{2}}$. Then $|\alpha| \leq 1$
and $|\beta| \leq 1$.  Let  $n$ be the number of those $F_{l}$'s
  for which $F_{l} \cap \partial U \neq \emptyset$.
 Then $n$ is a positive integer.  Because $a$ is a light point,
applying  Lemma~\ref{l:cru}  with $\alpha, \beta, \frac{1}{nk
C_{0}}$,  then there exists a function
 $g_{l}$ in $C(${\bf C}$_{\infty})$ that is analytic
 off $\overline{B(a,d_{l})}$ and satisfies:
1) $\|g_{l}\|\leq C_{1}$ ($C_{1}$ is a  positive universal
constant), 2) $g_{l}(\infty)=0$, 3) $g_{l}^{'}(\infty)=\alpha
d_{l}$, 4) $\beta(g_{l}, z_{l})= \beta d_{l}^{2}$, 5) $|\int g_{l}
x_{j} \hspace{.05in}d\mu| \leq \frac{1}{nk C_{0}}$ for all $ 1 \leq
j \leq k$.
Set $h_{l} = C_{0} g_{l}$. Then $h_{l}$ has the following
properties: 1) $  \|h_{l}\| \leq C_{0}C_{1}$, 2) $h_{l}$ is analytic
off $\overline{B(a,d_{l})}$, 3) $|\int h_{l}x_{j} d \mu| \leq
\frac{1}{nk}$ for $j\leq k$,
4) $h_{l}-f_{l}$ has a triple zero at $\infty$, that is,
$(h_{l}-f_{l})(\infty)=0$, $(h_{l}-f_{l})^{'}(\infty)=0$ and
$\beta(h_{l}-f_{l}, z_{l}) = 0$.

Let $\delta_{l}$ be the length of a side of $F_{l}$. Since $a\in
F_{l}$, it is evident that $B(a,d_{l})$ is contained in the square
with center $z_{l}$ and sides of length $ 2 \delta_{l}$.
  $f_{l} - h_{l}$ is clearly
analytic off $\{z: |z-z_{l}| \leq 2 \delta_{l} \}$.
 Since $f_{l}-h_{l}$ has a triple zeros at $\infty$,
$(z-z_{l})^{3}(f_{l} - h_{l})$ is also analytic off $\{z: |z-z_{l}|
\leq 2 \delta_{l} \}$. So the maximum principle implies that
\[ |(z-z_{l})^{3}(f_{l} - h_{l})| \leq
 2^{3}\delta_{l}^{3} \|f_{l}- h_{l}\|_{\infty}
\leq C_{0}(C_{1}+1)2^{3} \delta_{l}^{3} \hspace{.2in}
\mbox{whenever}\hspace{.05in} |z-z_{l}| \geq 2 \delta_{l}.
\]
Let $C_{2} = 8C_{0}(C_{1}+1)$. Then,  for each $z$ $ |f_{l}(z) -
h_{l}(z)|  \leq \mbox{min} (C_{2}, \frac{C_{2}\delta^{3}}
{|z-z_{l}|^{3}}). $ So it follows that for every $z$ (the first sum
is taken over those $l$'s for which $F_{l} \cap \partial U \neq
\emptyset$),
\begin{align*}
\sum |f_{l}- h_{l}| & \leq \sum_{l}^{\infty} \mbox{min} (C_{2},
\frac{C_{2}\delta^{3}} {|z-z_{l}|^{3}}) = \sum_{l}^{\infty}
\mbox{min} (C_{2}, \frac{C_{2}(\frac{5}{4})^{3}2^{-3k}}
{|z-z_{l}|^{3}})  \\& \leq C_{2} (\frac{5}{4})^{3} \mbox{min}\{1,
\frac{2^{-k}} {\mbox{dist}(z, \cup F_{l})}\}.
\end{align*}
Notice that  $f$ is analytic on those $F_{l}$'s for which
 $F_{l} \cap \partial U = \emptyset$. It follows  that
\[ f_{l} = T_{\phi_{l}} f = \frac{1}{\pi} \int
 \int \frac{f(z) - f(w)}{z-w} \frac{\partial
\phi_{l}(z)}{\partial \overline{z}} \hspace{.05in}d\mbox{ Area}
 = - \frac{1}{\pi} \int
 \int \frac{\partial f(z)}{\partial \overline{z}}\frac{\phi_{l}(z)}{z-w}
 \hspace{.05in}d\mbox{ Area}  =0
\]
for those $l$'s. So there are only finitely many $f_{l}$'s that are
not zero.

Now define $h_{l} =0$ if $l$ is such that $f_{l} =0$ and set $y_{k}
= f + \sum_{l}(h_{l} - f_{l})$. Then $y_{k} = \sum_{l}h_{l}$. For
any $z$ off $\partial U$,
 it is clear that
$ \mbox{dist} (z, \cup F_{l}) \rightarrow \mbox{dist} (z,\partial U)
\hspace{.05in} \mbox{as} \hspace{.05in} k \rightarrow \infty, $ and
hence it follows from the above inequalities that $y_{k} \rightarrow
f(z)$ for each $z$ in \\    {\bf C} $\setminus
\partial U$. According to 3),  we have
$ | \int y_{k} x_{j} \hspace{.05in}d\mu| \leq \frac{1}{k},
\hspace{.05in}\mbox{for} \hspace{.05in} 1\leq j\leq k. $ Notice that
\[|y_{k}| \leq |f| + |\sum_{l}(h_{l}-f_{l})|
                 \leq \|f\|_{\infty}+ C_{2}(\frac{5}{4})^{3},
\]
so $\{y_{k}\}$ is a bounded sequence. Since the weak-star topology
on the unit ball of the dual space of a separable Banach space is
metrizable, it follows by Alaoglu's theorem that  there exists a
subsequence $\{y_{k_{j}}\}$ that weak-star converges to some
 $h\in L^{\infty}(\mu)$. According to the last
inequality above,  we have that $ \int h x_{j} \hspace{.05in}d\mu =
0 \hspace{.05in}\mbox{ for all} \hspace{.05in}j\geq 1. $
Consequently, we have that $h\in A^{q}(K,\mu)$. Because $y_{k_{j}}
\rightarrow f$ pointwise off $\partial U$ and $f = 0$ off $\overline
U$, we have $h=0$ off $\overline U$.

Finally, we show that $\hat{h} = f$ on $U$. If $supp(\mu)$ contains
an open subset of $G$, then this is easy to see this is true (since
$y_{k} \rightarrow f$ pointwise on $U$ and hence $f = h = \hat{h}$
{\em a.e.} $[\mu]$ on $G$. Because both $f$ and $\hat{h}$ are
analytic on $U$, hence $f = \hat{h}$ on $U$). Otherwise, let $G$ be
open so that $\overline{G}\subset U$ and let $\rho = \mu+
Area|\overline{U}$.
Then  $\| \cdot   \|_{\mu}$ and  $\|\cdot  \|_{\rho}$ are equivalent
norms. By the definition of $abpe$, we see that $U\subset \nabla
A^{q}(K,\mu) = \nabla A^{q}(K,\rho). $ Therefore, there exists
$f_{1} \in A^{q}(K,\rho) \cap L^{\infty}(\rho)$ such that
$\hat{f}_{1} = f$ and $f_{1} = 0$ off  $\partial U$. Now set $h=
f_{1}|supp(\mu)$.
  We show  $h$ is the desired function.
Let $\{f_{n}\} \subset A(K)$ such that $f_{n} \rightarrow f_{1}$ in
$L^{q}(\rho)$. Then $f_{n}\rightarrow h$ in $L^{q}(\mu)$. Thus,
$f_{n}\rightarrow \hat{h}$ uniformly on $G$. Since $f_{n}\rightarrow
f$ uniformly on $G$ as well, it follows that $\hat{h} = f$ on $G$.
Because $U$ is  the union of such open subsets $G$,  we conclude
that $\hat{h} = f$ on $U$.

Now we consider a general $K$ that satisfies the hypothesis of this
lemma. Let $\Omega$ be the component of $K^{\circ}$ that contains
$U$. Then the multi-nicely connectivity of $\Omega$ insures that
there is circular domain $W$  and a conformal map $v$ from $\Omega$
onto $W$ such that $v$ is almost 1-1 on $\partial \Omega$ with
respect the harmonic measure of $\Omega$. Let $\mu= \mu_{0} +\tau$
be the decomposition such that
 $A^{q}(K,\tau)$ is pure and
$A^{q}(K,\mu) = L^{q}(\mu_{0}) \oplus A^{q}(K,\tau). $ By
Proposition~\ref{p:ABS2}, $\tau|\partial \Omega$ is absolutely
continuous with respect to the harmonic measure. Extend $v$ to
$\overline \Omega$ by defining its boundary values as its
nontangential limits and set $\nu=\tau\circ v^{-1}$. It is easy to
check that $v(U)$ is a component of $\nabla A^{q}(\overline W,\nu)$
and there is $h_{1}\in A^{q}(\overline W,\nu)$ such that
$\hat{h}_{1} =f\circ v^{-1}$.
Set $h=h_{1} \circ v^{-1}$. Then, $h\in A^{q}(\overline
\Omega,\tau)$ and it is straightforward to verify that $\widehat h=
f$. Extend $h$ to be a function on $K$ by defining $h=0$ off
$\overline \Omega$. By Lemma~\ref{l:key2} $h \in
A^{q}(\overline{\Omega},\mu|\overline{\Omega}) \subset A^{q}(K,\mu).
$ Clearly, $h$ does the job.

\end{proof}

\begin{lem} \label{l:EASY}
If\hspace{.06in}$a\in \nabla A^{q}(K,\mu)$, then
$\frac{f(z)-\hat{f}(a)}{z-a} \in A^{q}(K, \mu)$ for each $f\in
A^{q}(K,\mu)$.
\end{lem}

\begin{proof}
Let $W=\nabla A^{q}(K,\mu)$. Then there exists $\{f_{n}\} \subset
A(K)$ such that $f_{n} \rightarrow f$  in $L^{q}(\mu)$ and so $f_{n}
\rightarrow \hat{f}$ uniformly on compact subset of $W$.
 Thus $\frac{f_{n}(z) - f_{n}(a)} {z-a} \rightarrow
\frac{f(z) - \hat{f}(a)} {z-a} $
 uniform on a small
closed disk $ B(a,\delta) \subset W$. Note,
\begin{align*}
\int_{K}|\frac{f_{n}(z)-f_{n}(a)}{z-a}-\frac{f(z)-\hat{f}(a)}{z-a}
|^{q} d \mu
& 
   \leq
M\int_{K} |f_{n}(z) - f(z) - (f_{n}(a) -\hat{f}(a))|^{q} d \mu \\&
 + \int_{B(a,\delta)} |\frac{f_{n}(z) - f_{n}(a)} {z-a} - \frac{f(z) - \hat{f}(a)} {z-a} |^{q} d \mu,
\end{align*}
where $M=\sup_{z\in K\setminus B(a,\delta)}|\frac{1}{z-a}|^{q} $.
Thus, $\frac{f_{n}(z) - f_{n}(a)} {z-a} \rightarrow \frac{\hat{f}(z)
- \hat{f}(a)} {z-a} $ in $L^{q}(\mu)$. Since $\frac{f_{n}(z) -
f_{n}(a)} {z-a} \in A(K)$ for each $n$, the conclusion of the lemma
follows.

\end{proof}

\begin{lem} \label{l:ABPE}
Suppose that  $\nu \perp A(K)$ and $supp(\nu) \subset K$.
 Let $U$ be a component of $ K^{\circ} \setminus supp(\nu)$.
If $\hat{\nu}(a) \neq 0 $ at some $a\in U$,
 then $U\subset \nabla A^{1}(K,|\nu|)$.
\end{lem}

\begin{proof}
Clearly $\hat{\nu}(z) = \int \frac{1}{z-w} d\nu(w)$ is analytic on
$U$. Observe that for $f\in A(K)$,  $\frac{f(z)-f(a)}{z-a} \in A(K)$
for every $a\in K^{\circ}$.
 Suppose that  $\hat{\nu}(a) \neq 0$ for some $a\in U$. Then
there exists a small closed disk $\overline{B(a,\delta)} \subset U$
so that $\hat{\nu}(z) \neq 0$ on $\overline{B(a,\delta)}$.
 For each $\lambda \in B(a,\delta)$,
$\int \frac{f(z) - f(\lambda)}{z- \lambda} d \nu = 0$ and hence
\[ f(\lambda) = \frac{1}{\hat{\nu}(\lambda)} \int \frac{f(z)}{z-\lambda} d\nu,
\hspace{.05in} \mbox{for every}\hspace{.05in} f \in A(K).
\]
Since $\overline{B(a,\delta)}$ does not interest $supp(\nu)$, we see
that
\[|f(\lambda)| \leq c \|f\|_{L^{1}(|\nu|)}\hspace{.05in}
\mbox{ for some} \hspace{.05in}c > 0 \hspace{.05in}\mbox{
on}\hspace{.05in}
 \hspace{.05in}B(a,\delta).
\] Hence,
$a\in \nabla A^{1}(K,|\nu|)$.  Since the zeros of $\hat{\nu}$ is
isolated on $U$, it follows by Lemma~\ref{l:ISO}
 that $U \subset \nabla A^{1}(K,|\nu|)$.

\end{proof}

\begin{lem}
Let $h \in A^{q}(K,\mu)^{\perp}$ and set $\nu= h\mu$. Then $ \nabla
A^{1}(K,|v|) \subset \nabla A^{q}(K,\mu). $
\end{lem}
\begin{proof}
For $f\in A(K)$, by H\"{o}lder's inequality $   \|f\|_{L^{1}(|\nu|)}
\leq \|h\|_{L^{p}(|\nu|)} \|f\|_{L^{q}(\mu)} , $ where
$\frac{1}{q}+\frac{1}{p}=1$. The the conclusion of the lemma clearly
follows

\end{proof}

\begin{prop} \label{p:ABPE}
Let $\mu$ be a positive finite measure with $supp(\mu) \subset K$.
 Let $U$ be a component of $ K^{\circ}\setminus supp(\mu)$.
If\hspace{.05in} $U \cap \nabla A^{q}(K,\mu) \neq \emptyset$,
 then $U\subset \nabla A^{q}(K,\mu)$.
\end{prop}

\begin{proof}
Let $a\in U \cap \nabla A^{q}(K,\mu) $.  Then there is $g\in
A^{q}(K,\mu)^{\perp}$ such that $\int \frac{g d \mu}{z-a} \neq 0$.
Clearly,
\[g\mu \perp A(K)\hspace{.08in} \mbox{ and }\hspace{.08in}
U \subset K^{\circ}\setminus supp(g\mu).
\]
So it follows from the previous lemmas that
\[U \subset \nabla A^{1}(K,|g\mu|) \subset \nabla A^{q}(K,\mu).\]

\end{proof}

\begin{lem} \label {l:unique}
Let $\Omega = \nabla A^{q}(K,\mu)$ and let $U$ be a component of
$\nabla A^{q}(K,\mu)$. Suppose that $A^{q}(K,\mu)$ is pure. Let
$f\in A^{q}(K,\mu)$. If $\hat{f}=0$ and $f=0$  off\hspace{.03in}
 $\partial U$ \em{a.e.} $[\mu]$, then $f=0$.
\end{lem}

\begin{proof}
Since $ A^{q}(K,\mu)$ is pure, as we argued in the proof of
Proposition~\ref{p:ABS2}, there exists $h \in A^{q}(K,\mu)^{\perp}$
such that $h \neq 0$ {\em a.e.} $[\mu]$ and $h = 0 $ off
$\overline{K^{\circ}}$. So $h \perp
 A(\overline{K^{\circ}})$ as well.
Let $\nu = hf \mu$. Then $\nu$ is a measure such that it is
perpendicular to $A(K)$ and $supp(\nu) \subset \partial U$.
 We show  that $\hat{\nu}(a) =0$ off $\partial U$.

Let  $W$ be a component of $K^{\circ} \setminus supp(\nu)$. We claim
that $\hat{\nu}(z) = 0$ on $W$.  Suppose $\hat{\nu}(a) \neq 0$ for
some $a \in W$.  According to  Lemma~\ref{l:ABPE},
 \[W \subset \nabla A^{1}(K,|\nu|) \subset \nabla A^{q}(K,\mu).
\]
 By Lemma~\ref{l:EASY} and the hypothesis, we have that
$\frac{f}{z-a} \in A^{q}(K,\mu)$. But $h \perp A(K)$. So we conclude
that $\hat{\nu}(a) = 0$, which contradicting our assumption above.
Therefore, $\hat{\nu} = 0$ on $W$. In particular, $\hat{\nu} = 0$ on
$K^{\circ} \setminus \partial U$. It is easy to see that $
\overline{K^{\circ} \setminus \partial U} \supset
\overline{K^{\circ}}.$ \hspace{.02in}So, by the continuity we have
that $\hat{\nu} = 0$ on $\overline{K^{\circ}} - \partial U$. Because
$h \perp A(\overline{K^{\circ}})
 \supset R(\overline{K^{\circ}})$,
$\hat{\nu} = 0$ off $\overline{K^{\circ}}$. Hence, we conclude that
$\hat{\nu} = 0$ off $\partial U$.

Now, according to our definition, it is apparent that every
 point off $\partial U$ is a light point with respect to $|\hat{\nu}|$.
So it follows from  Lemma 3.7 in \cite{qiu1} that every point in
$\partial U$ is light as well. Consequently, every point in the
plane {\bf C} is a light point. Applying Theorem~\ref{t:light}, we
conclude that $v = hf\mu = 0$. Since $h \neq 0$ {\em a.e.} on $K$,
$f$ must be the zero function in $L^{q}(\mu)$. So we are done.

\end{proof}

\begin{lem} \label{l:finite}
Let $K$ be a compact subset in {\bf C} such that each component of
$K^{\circ}$ is finitely connected. Let $\mu$ be a positive finite
measure supported on $K$. Then each component $U$ of $\nabla
A^{q}(K,\mu)$ is a finitely connected domains conformally equivalent
to a circular domain in the plane. Moreover, the connectivity of $U$
 does not exceed the connectivity of the component of $K^{\circ}$
that contains $U$.
\end{lem}
\begin{proof}
Suppose $\nabla A^{q}(K,\mu) \neq \emptyset$. Let $U$ be a component
of $\nabla A^{q}(K,\mu)$ and  let $\Omega$ be the component of
$K^{\circ}$ that contains $U$. Let $M$ be the connectivity of
$\Omega$.

Now  suppose $F$ is a component of {\bf C}$\setminus U$. We claim
that $ F\cap (\mbox{ {\bf C}}\setminus  \Omega ) \neq \emptyset.  $
First, if $F$ is unbounded, this is obvious. So we assume $F$ is a
bounded subset in the plane and assume  that
 $ F\cap (\mbox{ {\bf C}} \setminus  \Omega ) = \emptyset$. Then
$F\subset \Omega$. Since $U$ is a connected domain, $F$ is
polynomially convex (this means the complement of $F$ is connected).
Since $U$ is finitely connected, there exists a Jordan curve
$\gamma$ in $U$ such that $F $ is contained in $V$, the bounded
Jordan domain enclosed by $\gamma$. Let $f\in A^{q}(K,\mu)$ and
choose a sequence of functions $\{r_{n}\}$ in $A(K)$ such that $
r_{n} \rightarrow f$ in $L^{q}(\mu)$. Since $\gamma$ is contained in
$ U \subset  \nabla A^{q}(K,\mu)$, it follows that $r_{n}
\rightarrow f$ uniformly on $\gamma$. Also, it is clear that we can
choose $\gamma$ such that  $dist(F, \gamma)$
 small enough that the closure of $V$
is contained in $\Omega \subset K^{\circ}$. Then each $r_{n}$ is
analytic on $\overline{V}$ and thus the maximum principle implies
that $\{r_{n}\}$ uniformly converges to a function $h$ near $F$.  By
the definition of $abpe$s, we have that $F \subset \nabla
A^{q}(K,\mu)$. But $F\cap \partial U \neq \emptyset$, and hence we
conclude that $\partial U \cap
 \nabla A^{q}(K,\mu)
\neq \emptyset$, which  is a contradiction.  Hence $ F\cap (\mbox{
{\bf C}} \setminus  \Omega ) \neq \emptyset$.

Now let $\{E_{i}\}$ be the collection of all the components of ({\bf
C} $ \setminus \Omega$) that intersect $F$. Then  $F \cup (\cup
E_{i})$ is connected compact subset and
 \[ [ F \cup (\cup E_{i})] \cap U = \emptyset. \]
So $F \cup (\cup E_{i}))$ is contained a component of {\bf C}
$\setminus U$ that contains $F$.  Hence $F \cup (\cup E_{i})  = F$
and therefore
 $E_{i} \subset F$ for each $i$.
Consequently, each component of {\bf C} $\setminus U$ contains at
least a component of ({\bf C} $\setminus \Omega$). Since the number
of the components of {\bf C} $\setminus \Omega$ is $M$, we see that
the number of the components of {\bf C} $\setminus U$  is less than
\mbox{or equal to $M$.}

Finally, since $U$ is finitely connected and since $\partial U$
contains no single-point component (by Lemma~\ref{l:ISO}), it
follows by a classical result \cite[Tsuji, p. 424]{tsuji} that $U$
is conformally equivalent to a circular domain.

\end{proof}

The next two propositions are Theorem 1 and Theorem 3 in
\cite{comm}, respectively. We include them for readers convenience
and self contained.

\begin{prop} \label{p:abs}
Let $A^{q}(K,\mu)$ be irreducible. Let $U =\nabla A^{q}(K,\mu)$ be a
finitely connected domain. If  the map $e$, defined by $e(f) =
\hat{f}$, from $A^{q}(K,\mu)\cap L^{\infty}(\mu)$ to $H^{\infty}(U)$
is surjective,
 then $\mu|\partial U \ll \omega_{U}$,
the harmonic measure of $U$.
\end{prop}

\begin{prop} \label{p:nont}
Let $A^{q}(K,\mu)$ be irreducible. Let $U =\nabla A^{q}(K,\mu)$ be a
finitely connected domain and let $u$ be a conformal map from a
circular domain $W$ onto $U$. If  the map $e$, defined by $e(f) =
\hat{f}$, from $A^{q}(K,\mu)\cap L^{\infty}(\mu)$ to $H^{\infty}(U)$
is surjective,
 then for each $f\in H^{\infty}(U)$
\[ e^{-1}(f)(a) = \lim_{z\rightarrow u^{-1}(a)} f\circ u(z) \hspace{.05in}
\mbox{a.e. on} \hspace{.05in}\partial U \hspace{.05in}\mbox{ with
respect to }\hspace{.05in}  \mu|\partial U.
\]
Moreover, $A(U) \subset A^{q}(K,\mu)$.
\end{prop}


\vspace{.05in} \noindent \underline{The proof of
Theorem~\ref{t:main}}.

Let $\mu= \mu_{0} +\tau$  be the decomposition such that
 $A^{q}(K,\tau)$ is pure and
\[A^{q}(K,\mu) = L^{q}(\mu_{0}) \oplus A^{q}(K,\tau).
\]
Suppose $A(K)$ is not dense in $L^{q}(\mu)$. Then $\tau \neq 0$ in
the decomposition.  According to Theorem~\ref{t:approx}, $\nabla
A^{q}(K,\tau) \neq \emptyset$.

 Let $\{U_{n}\}_{n=1}^{\infty}$
be the components of $\nabla A^{q}(K,\tau)$. For each $n\geq 1$, by
Lemma~\ref{l:scheme}
 there exists $f_{n}$ in $A^{q}(K,\tau)\cap L^{\infty}(\tau)$
such that $\hat{f_{n}} = \chi_{U_{n}}$ and $f_{n} = 0$ off $U_{n}$.
Since $U_{n}$'s are pairwise disjoint, we have
 $\widehat{f_{n}f_{m}} =\hat{f}_{n} \hat{f}_{m} = 0.
$ It follows by Lemma~\ref{l:unique} that $f_{n}f_{m} = 0$.
Similarly, since $\hat{f}_{n}^{2} = \hat{f_{n}}$, we get that $
f^{2}_{n} = f_{n}.$ Therefore,  we conclude that
$f_{n}=\chi_{\Delta_{n}}$ for some Borel subset $\Delta_{n}$. But
$\tau(\Delta_{n} \cap \Delta_{m}) = 0$ (because $f_{n}f_{m}=0$).
 Thus $\Delta_{n}$'s can be chosen to be pairwise disjoint.
Moreover, since $f_{n} = 0$ off $\overline U_{n}$,
 we can also require that $\overline \Delta_{n} \subset \overline U_{n} $.

For each $n\geq 1$, let $K_{n}=\overline{U}_{n}$ and let $\mu_{n} =
\tau|\Delta_{n}$. We claim that $U_{n} = \nabla
A^{q}(K_{n},\mu_{n}). $ Let $\Omega$ be the component of $K^{\circ}$
that contains $U_{n}$. By Lemma~\ref{l:key2},
$A^{q}(\overline{\Omega},\tau|\overline{ \Omega})\subset
 A^{q}(K,\tau)$. Note,  every function $f$ in $A^{q}(\overline{\Omega},
\tau|\overline{ \Omega})$ has zero values off $ \overline \Omega$.
Clearly, $f_{n} $ ($\chi_{\Delta_{n}})$  belongs to $
A^{q}(\overline{\Omega},\tau|\overline{ \Omega})$ also.  Set $F_{n}=
U_{n} \cup (\Delta_{n}\setminus U_{n}). $ Because $\overline
\Delta_{n}  \subset \overline U_{n} $, we see  that $\chi_{F_{n} } =
f_{n} \in A^{q}(\overline{\Omega},\tau|\overline{ \Omega})$. So by
Lemma~\ref{l:elem1}, $\chi_{F_{n}}f  \subset A^{q}
(\overline{\Omega}, \tau|\overline{\Omega})$
 for each $f\in A(\overline{\Omega}) $. This implies that
$ A^{q}(\overline{\Omega},\mu_{n})=
 A^{q}(\overline{\Omega},\tau|F_{n} ) \subset A^{q}(\overline{\Omega},\tau).
$ Let $a \in U_{n}$.  Then there exists $c>0$ and an open open disk
 $D_{a} \subset U_{n}$ such that for all
$f\in A(\overline \Omega)$
\[|f(a)| \leq c\{\int |f|^{q} d \tau \}^{\frac{1}{q}}, \hspace{.05in}
a \in D_{a}.  \] Let $g \in A(\overline \Omega)$.
Then there is a sequence $\{q_{i}\}_{i=1}^{\infty}$ in $A(\overline
\Omega)$ so that $q_{i}\rightarrow g\chi_{F_{n}}$ in $L^{q}(\tau)$.
 Then $q_{i}\rightarrow g$ uniformly on $D_{a}$. Hence, it follows that
for  $a\in D_{a}$
\[|g(a)| =
\lim|q_{i}(a)| \leq c  \lim \{\int|q_{i}|^{q} d\tau\}^{\frac{1}{q}}=
c\{\int g\chi_{F_{n}} d\tau \}^{\frac{1}{q}} = \hspace{.05in} c\{
\int |g|^{q} \hspace{.05in}d\mu_{n} \}^{\frac{1}{q}}.
\]
Thus $ a\in \nabla A^{q}(\overline{\Omega},\mu_{n})$. Therefore,
$U_{n}\subset \nabla A^{q}(\overline{\Omega},\mu_{n})$. By
Lemma~\ref{l:key2}, $\nabla A^{q}(\overline \Omega,\mu_{n}) \subset
\nabla A^{q}(K,\tau)$, so we see (notice that $U_{n}$ is a component
of $\nabla A^{q}(K,\tau)$)
 that $ U_{n} = \nabla A^{q}(\overline \Omega,\mu_{n}). $

The hypothesis and Lemma~\ref{l:finite} together imply that  $U_{n}$
is a finitely connected domain. It is also easy to see that
$A^{q}(\overline \Omega,\mu_{n})$ is irreducible. Applying
Proposition~\ref{p:nont}, we have $A(U_{n}) \subset A^{q}(\overline
\Omega, \mu_{n})$. Consequently, $A^{q}(\overline U_{n},\mu_{n})
\subset A^{q}(\overline \Omega, \mu_{n})$. Therefore, we conclude
that $A^{q}(\overline \Omega, \mu_{n})= A^{q}(\overline
U_{n},\mu_{n}) $.  Hence, $U_{n} = \nabla A^{q}(K_{n},\mu_{n}).$
This proves the claim.

Since each  $A^{q}(K_{n},\mu_{n})$  is contained in $A^{q}(K,\mu) $
and since $\{A^{q}(K_{n},\mu_{n})\}$ are pairwise othrogonal, we
have
\[ A^{q}(K,\mu) \supset L^{q}(\mu_{0}) \oplus
A^{q}(K_{1}, \mu | \Delta_{1}) \oplus ... \oplus
 A^{q}(K_{n}, \mu | \Delta_{n}) \oplus ...  .
\]
For the other direction of the equality, let $f$ be the pointwise
limit of $\{ \sum_{i=1}^{n} f_{i} \}_{n=1}^{\infty}$. Then the
bounded convergence theorem
implies that $f\in A^{q}(K,\tau)$. We show that $1 - f = 0$ a.e.
$[\tau]$. Otherwise, there exists a Borel subset $E$ of the support
of $\tau$ such that $0 \neq \chi_{E} =  1 -  f$. Since both $1$ and
$f$ are in $A^{q}(K,\tau)$, we have that $\chi_{E} \in
A^{q}(K,\tau)$. By the purity, we have $L^{q}(\tau|E) \neq
A^{q}(K,\tau|E)$. So it follows by Theorem~\ref{t:approx} that
$\nabla A^{q}(K,\tau|E) \neq \emptyset$. But $\chi_{E} f_{n} = 0$
for each $n \geq 1$, thus we have $ \nabla A^{q}(K,\tau|E) \cap
(\cup U_{n}) = \emptyset. $ But by the definition of $abpe$s,
$\nabla A^{q}(K,\tau|E) \subset \nabla A^{q}(K,\mu) = \cup U_{n}. $
This is a  contradiction, and hence $f - 1 =0$. Therefore, $\{
\Delta_{n}\}$ is a Borel partition. Let $g\in A^{q}(K,\tau)$.
 By the Lebesgue dominated convergence theorem, we conclude
  that $g= \lim_{n\rightarrow \infty}
\sum_{i=1}^{n} f_{i} g$ in $ L^{q}(\tau)$. Therefore,
\[A^{q}(K,\mu) \subset L^{q}(\mu_{0})
 \oplus A^{q}(K_{1}, \mu | \Delta_{1}) \oplus ... \oplus
 A^{q}(K_{n}, \mu | \Delta_{n}) \oplus ...  .
\]
Consequently,
\[A^{q}(K,\mu) = L^{q}(\mu_{0})
 \oplus A^{q}(K_{1}, \mu | \Delta_{1}) \oplus ... \oplus
 A^{q}(K_{n}, \mu | \Delta_{n}) \oplus ...  .
\]

Now we  prove the rest of Theorem~\ref{t:main}: For 1), since  we
have already proved $\overline U_{n} \subset \Delta_{n}$ above, we
only need to show that $A^{q}(K,\mu_{n}) =A^{q}(\overline
U_{n},\mu_{n})$. Because $A(K)\subset A(\overline U_{n})$, it
follows by the definition of $abpe$ that
\[U_{n} \subset \nabla A^{q}(\overline U_{n},\mu_{n}) \subset
 \nabla A^{q}(K,\mu_{n}).
\]
Notice that  $A^{q}(K,\mu_{n})\subset A^{q}(\overline
U_{n},\mu_{n})$ and the latter is irreducible. So we see that
$A^{q}(K,\mu_{n})$ is irreducible also. This implies that $\nabla
A^{q}(K,\mu_{n})$ have only one component. Also, it is clear that
$\nabla A^{q}(K,\mu_{n}) \subset \nabla A^{q}(K,\tau)$. So we
conclude that $\nabla A^{q}(K,\mu_{n}) =U_{n}$. Thus, there is
$h_{n}\in A^{q}(K,\mu_{n})$ such that $\hat{h}_{n} = \chi_{U_{n}}$
and $h_{n} =0 $ off $\overline  U_{n}$. By the uniqueness, we get
that  $h_{n} = f_{n} = \chi_{\Delta_{n}}$.
 As we proved above,
$f = \sum_{n=1}^{\infty} fh_{n}$ for each $f\in A^{q}(K,\tau)$.
Hence, we conclude that
 \begin{align*}
 A^{q}(K,\tau) & \subset  A^{q}(K, \mu_{1}) \oplus ...
\oplus A^{q}(K, \mu_{n}) \oplus ...
  \\ &  \subset A^{q}(\overline U_{1}, \mu_{1}) \oplus ... \oplus
 A^{q}(\overline U_{n}, \mu_{n}) \oplus ...
\\ & = A^{q}(K,\tau).
\end{align*}
Consequently, $A^{q}(K,\mu_{n}) = A^{q}(\overline U_{n},\mu_{n})$
for each $n \geq 1$.

2) follows from Lemma~\ref{l:finite}.

For 3),  let $e$ be the map, $f \rightarrow \hat{f}$,
 from $L^{\infty}(\mu_{n})
\cap A^{q}(K_{n},\mu_{n})$ into $H^{\infty}(U_{n})$. Then $e$ is
surjective by Lemma~\ref{l:scheme} and is injective by
Lemma~\ref{l:unique}.
 Since $e(fg) = e(f) e(g)$, $e$ is
an algebraic isomorphism between two commutative Banach algebras
 and thus $e$ is an isometry.

Next we need to show that  $e$ is a weak-star homeomorphism. To do
this, we will argue as in \cite{ce}. Using {\em Krein-Smulian
theorem} it suffices to show that $e$ is weak-star sequentially
continuous.

Recall that  a sequence of functions in $H^{\infty}(U_{n})$ is
weak-star Cauchy sequence if and only if it is uniformly bounded on
$U_{n}$ and it is a Cauchy sequence in the topology of pointwise
convergence.  Let  $\{h_{i}\}$  be a sequence in
$A^{q}(K_{n},\mu_{n})\cap L^{\infty}(\mu_{n})$ that converges to
zero in the weak star topology. By the  uniform boundedness,
$\{h_{i}\}$ is bounded and hence $\{e(h_{i})\}$ is also bounded.
 Let $a\in U_{n}$ and let $k_{a}$ be kernel
function. Then
\[ \lim_{i\rightarrow  \infty} e(h_{i})
= \lim_{i\rightarrow \infty} \hat{h}_{i}(a)
 = \lim_{i\rightarrow \infty}
 \int h_{i} k_{a} \hspace{.02in} d\mu_{n} =0.
\]
So  $e(h_{i})$ weak-star converges to zero. Therefore,  $e$ is a
weak-star homeomorphism.

4) follows from Proposition~\ref{p:abs} and
Proposition~\ref{p:nont}.

For 5), by Proposition~\ref{p:nont}, each $f\in H^{\infty}(U_{n})$
has nontangential limits almost everywhere on $\partial U_{n}$ with
respect to $\mu|\partial U_{n}$ and the nontangential limits are
equal to $e^{-1}(f)$ {\em a.e.} $[\mu|\partial U_{n}]$. Since
$\hat{f}_{e} = f = f^{*} $ {\em a.e.} $[\mu]$ on $U_{n}$, we see
that $f^{*}= e^{-1}(f) |\Delta_{n}$ {\em a.e.} $[\mu]$. Evidently
\begin{align*}
   m(e(f)) = m(\hat{f}) & = \left\lbrace
  \begin{array}{c l}
    \hat{f} \hspace{.05in}\mbox{on}\hspace{.05in}U_{n}  &\\
    e^{-1}(\hat{f}) \hspace{.05in}\mbox{on}\hspace{.05in}\partial U_{n}
  \end{array}
\right.
 \\ & = \left\lbrace
  \begin{array}{c l}
    f \hspace{.05in}\mbox{on}\hspace{.05in}U_{n}   &\\
    f \hspace{.05in}\mbox{on}\hspace{.05in}\partial U_{n}
  \end{array}
\right.
\\ & = f,
 \hspace{.2in}\mbox{for each}\hspace{.05in} f\in A^{q}(K_{n},\mu_{n}).
\end{align*}
Therefore, $m$ is the inverse map of $e$. So the proof of
Theorem~\ref{t:main} is complete.

\vspace{.06in}
 \noindent{\bf Remark:} This paper contains the best and
close-up result in an once quite active research area. However, this
paper has not been cited by any other authors but me, while
\cite{jm} has been cited more than 80 times. I re-published this
paper on Arxiv to hope to bring more attentions from future
generation of mathematicians to the results in this paper, which I
believe is great and can stay in history of mathematics.

 \vspace{.05in} \noindent {\bf
Acknowledgement.} During 1994, I gave several talks on the result of
this work in turn in the SouthEast Analysis Meeting at Virginia
Tech, UC-Berkeley, AMS Summer Research Conference at Mt. HolyOak,
Brown University and Wabash Conferences at IUPUI. When $R(K)$ is a
hypodirichlet algebra, A slightly simper version of
Theorem~\ref{t:main} was stated as the last theorem in my work
\cite{qiu1}.

\end{document}